\documentclass[10pt,oneside]{amsart}  
\usepackage[left=3.1cm,right=3.1cm,top=3.1cm,bottom=3.1cm]{geometry}
\usepackage{graphics}
\usepackage{graphicx}
\usepackage{multicol}
\usepackage{epsfig}
\usepackage{amsmath,amssymb,amsfonts, amsthm,epsfig}
\usepackage{verbatim}
\usepackage{enumerate}
\usepackage{hyperref}

\newtheorem{thm}{Theorem}

\newtheorem{lemma}{Lemma}
\newtheorem{rek}[lemma]{Remark}
\newtheorem{prop}[lemma]{Proposition}
\newtheorem{defi}[lemma]{Definition}

\newtheorem{cor}[lemma]{Corollary}

\numberwithin{lemma}{section}
\newcommand{\be}{\begin{equation}}
\newcommand{\en}{\end{equation}}
\newcommand{\ve}{\mathbf}

\renewcommand{\l}{{\lambda}}
\newcommand{\R}{{\mathbb{R}}}
\newcommand{\T}{{\mathbb{T}}}

\newcommand{\N}{{\mathbb{N}}}
\newcommand{\mattwo}[4]
	{\left(\begin{array}{cc}
                        #1  & #2   \\
                        #3 &  #4
                          \end{array}\right) }
\newcommand{\matthree}[9]
{\left(\begin{array}{ccc}
                        #1  & #2 & #3  \\
                        #4 &  #5 & #6 \\
                        #7 &  #8 & #9
                          \end{array}\right) }
\newcommand{\p}{\partial}

\newcommand{\Aonel}{\mathcal{A}_1^\l}

\newcommand{\Ail}{\mathcal{A}_i^\l}
\newcommand{\Ais}{\mathcal{A}_i^\sigma}
\newcommand{\Aoneln}{\mathcal{A}_{1,n}^\l}
\newcommand{\Aonez}{\mathcal{A}_1^0}
\newcommand{\Aonezn}{\mathcal{A}_{1,n}^0}
\newcommand{\Atwol}{\mathcal{A}_2^\l}
\newcommand{\Atwoln}{\mathcal{A}_{2,n}^\l}
\newcommand{\Atwoz}{\mathcal{A}_2^0}
\newcommand{\Atwozn}{\mathcal{A}_{2,n}^0}
\newcommand{\Bl}{\mathcal{B}^\l}
\newcommand{\Bs}{\mathcal{B}^\sigma}
\newcommand{\Bln}{\mathcal{B}^\l_n}
\newcommand{\Bz}{\mathcal{B}^0}
\newcommand{\Bzn}{\mathcal{B}^0_n}
\newcommand{\Cl}{\mathcal{C}^\l}
\newcommand{\Cln}{\mathcal{C}^\l_n}

\newcommand{\Dl}{\mathcal{D}^\l}
\newcommand{\Dln}{\mathcal{D}^\l_n}

\newcommand{\Ml}{\mathcal{M}^\l}

\newcommand{\Ql}{\mathcal{Q}^\l}
\newcommand{\Qs}{\mathcal{Q}^\sigma}

\newcommand{\Konez}{\mathcal{K}_{1}^0}
\newcommand{\Konezn}{\mathcal{K}_{1,n}^0}

\newcommand{\Mz}{\mathcal{M}^0}
\newcommand{\Ms}{\mathcal{M}^\sigma}
\newcommand{\Mnz}{\mathcal{M}^0_n}
\newcommand{\Mlz}{\mathcal{M}^{\l_0}}
\newcommand{\Mnl}{\mathcal{M}^\l_n}

\newcommand{\Fnz}{\mathcal{F}^0_n}

\newcommand{\Mnln}{\mathcal{M}^{\l_n}_n}

\newcommand{\lptz}{L_{P,0}^2}
\newcommand{\lpt}{L_P^2}
\newcommand{\lwt}{L_\pm^2}
\newcommand{\la}{\lambda}
\newcommand{\La}{\Lambda}
\newcommand{\Proj}{\mathcal{P}}

\newcommand{\hpt}{H_P^2}
\newcommand{\hptz}{H_{P,0}^2}

\begin{document}

\title{Instability of Nonsymmetric Nonmonotone Equilibria of the Vlasov-Maxwell System}
\author{Jonathan Ben-Artzi}
\today
\email{yonib@math.brown.edu}
\address{Department of Mathematics\\Brown University\\Providence, RI 02912}

\begin{abstract}
We consider the $1\frac{1}{2}$-dimensional relativistic Vlasov-Maxwell system that describes the time-evolution of a plasma. We find a relatively simple criterion for spectral instability of a wide class of equilibria. This class includes non-homogeneous equilibria that need not satisfy any additional symmetry properties (as was the case in previous results), nor should they be monotone in the particle energy. The criterion is given in terms of the spectral properties of two Schr\"{o}dinger operators that arise naturally from Maxwell's equations. The spectral analysis of these operators is quite delicate, and some general functional analytic tools are developed to treat them. These tools can be applied to similar systems in higher dimensions, as long as their domain is finite or periodic.
\end{abstract}

\maketitle

\section{Introduction}

\subsection{The relativistic Vlasov-Maxwell system}
In this paper we consider the linear stability of a super-heated plasma. The plasma is assumed to have low density, and thus collisions between particles may be ignored. We consider a neutral plasma consisting of two species -- ions and electrons -- differentiated by $\pm$ superscripts. The behavior of such a system is governed by the relativistic Vlasov-Maxwell (RVM) system of equations (see, e.g. \cite{glassey})
	\begin{subequations}\label{mainequations}
	\begin{equation}\label{mainvlasov}
	\p_tf^{\pm}+\hat{v}^\pm\cdot\nabla_xf^{\pm}+\frac{q^\pm}{m^\pm}\left(\ve{E}+\ve{E}^{ext}+\frac{\hat{v}^\pm}{c}\times\left(\ve{B}+\ve{B}^{ext}\right)\right)\cdot\nabla_vf^{\pm}=0,
	\end{equation}
	\begin{equation}
	\p_t\ve{E}=c\nabla\times\ve{B}-\ve{j},
	\hspace{10pt}
	\nabla\cdot\ve{E}=\rho,
	\hspace{10pt}
	\p_t\ve{B}=-c\nabla\times\ve{E},
	\hspace{10pt}
	\nabla\cdot\ve{B}=0,
	\end{equation}
	\begin{equation}
	\rho=4\pi\int(q^+f^++q^-f^-)\;dv,
	\hspace{10pt}
	\ve{j}=4\pi\int(\hat{v}^+q^+f^++\hat{v}^-q^-f^-)\;dv,
	\end{equation}
	\end{subequations}
where $c$ is the speed of light, $v$ is the momentum, $\hat{v}^\pm:=v/\sqrt{(m^\pm)^2+|v|^2/c^2}$ are the velocities of both species, $q^\pm$ are the charges and $m^\pm$ are the masses of the two species. The transport equation \eqref{mainvlasov} is called the Vlasov equation, and it is coupled with Maxwell's equations. $\ve{E}(t,x)$ and $\ve{B}(t,x)$ are the electric and magnetic fields, $f^{\pm}(t,x,v)\geq0$ are the electron and ion distribution functions, and $\ve{E}^{ext}$ and $\ve{B}^{ext}$ are external electric and magnetic fields. In addition, $\rho(x)$ and $\ve{j}(x)$ are the charge and current densities at the point $x$, respectively.

This paper extends the results of the author in \cite{benartzi1} where the equilibrium was nonmonotone in the particle energy, but certain symmetries of the equilibrium state were assumed in order to make the analysis simpler. Nearly all these assumptions are dropped, and only some of the decay assumptions are kept. The results in this paper are more general, and reduce to the previously known results when the aforementioned symmetries \emph{are} assumed. For most of the technical lemmas in this paper, we shall refer the reader to \cite{benartzi1} instead of repeating the proofs here.

Mathematically, one of the main improvements of this paper is the general functional analytic setting in which we present the results. The machinery we introduce (in \S\ref{func-anal-back}) helps us gain a better understanding of the behavior of spectra of families of truncated operators (that is, operators that depend upon \emph{two} parameters) that do not necessarily vary continuously in the operator-norm topology.

The main physical improvement of this paper is in the dropping of numerous assumptions that previously appeared in \cite{benartzi1} and in \cite{rvm1,lin-strauss-nl,rvm2}. In the latter works, the authors always assumed a strictly monotone decrease in plasma density as a function of the energy (see \eqref{eq:monotone} for a precise definition). While such an assumption is physically reasonable, one can certainly envision physical systems where this assumption does not hold. For this matter, the result of Penrose \cite{penrose} is relevant: in a simpler spatially homogeneous model, he showed that certain ``nonmonotone'' equilibria are linearly stable. In \cite{benartzi1} the author dropped the monotonicity assumption, but made certain symmetry assumptions (discussed in \S\ref{1.5-dim-sec}) on the equilibrium state. Those assumptions are relevant for a physical situation in which there are strong magnetic fields and negligible electric fields. Dropping all these assumptions permits applications to a wide range of equilibrium states.

Some more recent mathematical results that are concerned with stability of equilibrium solutions of Vlasov systems include \cite{guo1,guo2,batt-rein2,rein-stability}. In the electrostatic case, Mouhot and Villani \cite{mouhot-villani} recently established the phenomena known as \emph{Landau damping} mathematically. Vlasov systems may also describe stellar dynamics, when one considers an attractive potential. See \cite{guo-rein} for example.

\subsection{The $1\frac{1}{2}$ Dimensional Case}\label{1.5-dim-sec}
For simplicity, we take an equilibrium of the lowest dimensional system which has a nontrivial magnetic field: the so-called one-and-one-half-dimensional case. It turns out that this is the right symmetry to consider when studying \emph{tokamaks}. In this setting, we have one space dimension and a two-dimensional momentum space. The single spatial variable $x$ corresponds to $v_1$, and the additional velocity dimension is denoted by $v_2$. We write $v=(v_1,v_2)$. To simplify our notation, we take all physical constants and masses to be $1$, and we take the charges to be $+1$ and $-1$. The notation is as follows: we let $f^\pm(t,x,v)$ be the electron and ion distribution functions, and $\ve{E}(t,x)=\left(E_1(t,x),E_2(t,x),0\right)$ and $\ve{B}(t,x)=(0,0,B(t,x))$ be the electric and magnetic fields. In addition, we define the electric and magnetic potentials $\phi$ and $\psi$, which satisfy
	\begin{equation}\label{potentials}
	\p_x\phi
	=
	-E_1
	\hspace{2 cm}
	\p_x\psi
	=
	B.
	\end{equation}

We assume the existence of an equilibrium $f^{0,\pm}(x,v)$ which is a solution of RVM. Existence is guaranteed by the result of Glassey and Schaeffer \cite{gsc1.5} and some explicit examples were constructed in \cite{benartzi1,rvm1}. (See also \cite{gsc2.5} for a similar existence result in a higher dimensional setting.) By Jeans' Theorem (cf \cite{dendy}) this equilibrium can be represented in the coordinates (invariants of the particle equations)
	$$
	e^\pm
	=
	\left<v\right>\pm\phi^0(x)\hspace{1cm}
	p^\pm
	=
	v_2\pm\psi^0(x)
	$$
as
	$$
	f^{0,\pm}(x,v)=\mu^\pm(e^\pm,p^\pm),$$
where $\phi^0$ and $\psi^0$ are the equilibrium electric and magnetic potentials, satisfying $E_1^0=-\p_x\phi^0$ and $B^0=\p_x\psi^0$, with $E^0_1$ and $B^0$ being the equilibrium electric and magnetic fields. In addition $E_2^0\equiv0$. Henceforth it will be understood that $\mu^\pm$ are evaluated at $(e^\pm, p^\pm)$, so we will simply write $\mu^\pm(e,p)$. In this paper, we consider the ``\emph{nonmonotone}'' case. By that we simply mean that
	\begin{equation}\label{eq:monotone}
	\mu^\pm_e:=\frac{\p\mu^\pm}{\p e}\nless0
	\end{equation}
on some subset of the set $\{\mu^\pm(e,p)>0\}$. Here, ``$\mu_e^\pm$" means ``the derivative of $\mu^\pm$ with respect to the first component evaluated at $(e^\pm,p^\pm)$." Similarly, ``$\mu_p^\pm$" is the derivative with respect to the second component. By ``monotone", we mean that $\mu^\pm_e<0$ on the set $\{\mu^\pm>0\}$. Roughly speaking, the coordinates $e$ and $p$ should be understood to be energy and momentum respectively. The monotone case was the subject of \cite{rvm1,lin-strauss-nl,rvm2}. In \cite{benartzi1} the author dropped the monotonicity assumption, but imposed a so-called \emph{purely magnetic} assumption, where one assumes that $\phi^0\equiv0$, as well as a certain symmetry assumption on the two species: $\mu^+(e,p)=\mu^-(e,-p)$. Both assumptions simplified parts of the analysis, and both are dropped in this paper. We therefore note that in this paper there is no prescribed relationship between the electrons and the ions.

\subsection{Main Results}\label{mainres}
We assume for simplicity that the equilibrium has some given period $P$. We define three operators acting on functions of $x$, whose properties will be rigorously treated later:
	\begin{eqnarray*}
	\Aonez h&=&-\p_x^2h-\left(\sum_\pm\int\mu^\pm_e\;dv\right)h+\sum_\pm\int\mu^\pm_e\Proj^\pm h\;dv,\\
	\Atwoz h&=&-\p_x^2h-\left(\sum_\pm\int\hat{v}_2\mu^\pm_p\;dv\right)h-\sum_\pm\int\mu^\pm_e\hat{v}_2\Proj^\pm (\hat{v}_2h)\;dv,\\
	\Bz h&=&\left(\sum_\pm\int\mu^\pm_p\;dv\right)h+\sum_\pm\int\mu^\pm_e\Proj^\pm(\hat{v}_2h)\;dv.
	\end{eqnarray*}
where $\Proj^\pm$ are projection operators onto some subspaces of a certain Hilbert space (see Definition \ref{proj-define}). The operators $\Aonez$ and $\Atwoz$ have as their domain the space
	$$
	\hpt
	=
	\left\{h\text{ is }P\text{ periodic on }\R\text{ and }h\in H^2(0,P)\right\}
	$$
where $H^2(0,P)$ is the usual Sobolev space of functions on $(0,P)$ whose first two derivatives are square integrable. The domain of $\Bz$ is
	$$
	\lpt
	=
	\left\{h\text{ is }P\text{ periodic on }\R\text{ and }h\in L^2(0,P)\right\}.
	$$
In \cite{benartzi1} this operator turned out to be trivial, due to cancellations that came from the symmetry assumptions and the ``purely magnetic'' assumption. This fact was exploited in the analysis of the matrix operator defined in \eqref{mz} which turned out to be diagonal. This is not the case here.

We also define the number $l^0$ as:
	$$
	l^0=\frac{1}{P}\sum_\pm\int_0^P\int\hat{v}_1\mu^\pm_e\Proj^\pm\left(\hat{v}_1\right)\;dv\;dx.
	$$
In order to properly define function spaces that include functions that do not necessarily decay at infinity, we must define appropriate weight functions. This definition is different from the one found in \cite{benartzi1} due to the appearance of two distinct energies $e^+$ and $e^-$ (which was not the case in \cite{benartzi1} due to the ``purely magnetic'' assumption). Define
	\begin{equation}\label{weight}
	w^\pm(e^\pm)
	=
	c(1+|e^\pm|)^{-\alpha}
	\end{equation}
for some $\alpha>2$ and $c>0$. We require $\mu^\pm\in C^1$ to satisfy
	\begin{equation}\label{weight2}
	(|\mu^\pm_e|+|\mu^\pm_p|)(e^\pm,p^\pm)
	\leq
	w^\pm(e^\pm)
	\end{equation}
so that $\int(|\mu^\pm_e|+|\mu^\pm_p|)\;dv$ is finite. Obviously, we require $\mu^\pm(e,p)\geq0$.

In this paper, we denote
	$$
	neg(\mathcal{F})
	=
	\left\{
	\text{the number of negative eigenvalues (counting multiplicity) of the operator }\mathcal{F}
	\right\}.
	$$
Similarly, $pos(\mathcal{F})$ denotes the number of positive eigenvalues, and $z(\mathcal{F})$ denotes the dimension of the kernel of $\mathcal{F}$. Later we will show that $\Aonez$ and $\Atwoz$ have discrete spectra and finitely many negative eigenvalues. We also define:

\begin{defi}[Spectral instability]
We say that a given equilibrium $\mu^\pm(e,p)$ is \emph{spectrally unstable}, if the system linearized around it has a purely growing mode solution of the form
	\begin{equation}\label{purely}
	\left(e^{\la t}f^\pm(x,v),e^{\la t}E(x),e^{\la t}B(x)\right),
	\hspace{1cm}
	\la>0.
	\end{equation}

\end{defi}

\begin{thm}\label{mainthm}
Let $f^{0,\pm}(x,v)=\mu^\pm(e,p)$ be a periodic equilibrium satisfying \eqref{weight2}. Assume that the null space of $\Aonez$ consists of the constant functions and that $l^0\neq0$. Then the equilibrium is spectrally unstable if
	\be\label{mainthm-condition}
	neg\left(\Atwoz+\left(\Bz\right)^*\left(\Aonez\right)^{-1}\Bz\right)> neg\left(\Aonez\right)+neg(-l^0).
	\en
\end{thm}

\begin{thm}\label{mainthm2}
Under the additional assumption that the null space of $\Atwoz$ is trivial, the equilibrium is spectrally unstable if
	\be\label{mainthm2-condition}
	neg\left(\Atwoz+\left(\Bz\right)^*\left(\Aonez\right)^{-1}\Bz\right)\neq neg\left(\Aonez\right)+neg(-l^0).
	\en
\end{thm}

The summary of the proofs of these theorems can be found on page \pageref{proof:theorems}.\\


This paper is organized as follows: in \S\ref{func-anal-back} we establish certain limiting spectral properties of Schr\"{o}dinger operators that depend upon two parameters and have merely strong limits (but no limits in the operator-norm topology). In \S\ref{setup-section} we reduce our problem to an equivalent selfadjoint problem \eqref{main-mat-op}. This problem depends upon the parameter $0<\l<\infty$ (see \eqref{purely}), which measures the rate of the exponential growth of the unstable mode. Our goal is to employ a continuation method introduced by Z. Lin \cite{lin01} in which one compares the spectrum when $\l=0$ and when $\l$ tends to $\infty$. However, since it is difficult to keep track of the spectrum, in \S\ref{large-small} we follow a truncation technique first introduced in \cite{rvm2}. As opposed to \cite{rvm2}, in this paper we use this technique to perform a \emph{full} truncation of the infinite-dimensional problem, making it finite-dimensional. In \S\ref{limit-section} we apply the results of \S\ref{func-anal-back} in order to find a solution to the approximate finite-dimensional problem (in Lemma \ref{approx-matrix-op-lemma}). The main difficulty is in obtaining information about the spectrum when $\l=0$. Finally, we retrieve the original problem by showing that the approximate solutions have a limit, as we let the truncation parameter tend to $\infty$ (see Lemma \ref{untrunc-problem}). This provides a growing mode as in \eqref{purely}.

\section{Functional Analysis Setup}\label{func-anal-back}
In this section we introduce some results about convergence of operators that depend upon two parameters. We state these results in a general functional analytic setting. Throughout this section, $C$ will be used to denote a generic constant that will change from equation to equation. We begin with some elementary comments. Our key result is Proposition \ref{uniform} and the propositions that follow.\\

We say that a sequence of operators $T_n$ in some Hilbert space $\mathcal{H}$ converge to another operator $T$ \emph{strongly} if for every $u\in \mathcal{H}$
	$$
	\|T_nu-Tu\|\to0
	$$
as $n\to\infty$, where $\|\cdot\|$ is the norm on $\mathcal{H}$. In this case we write $T_n\xrightarrow{s} T$. 
\begin{lemma}\label{uniform-bound}
Let $A,B$, $\{A_n\}_{n=1}^\infty,\{B_m\}_{m=1}^\infty$ be bounded operators on $\mathcal{H}$. Assume that $A_n \xrightarrow{s} A$ and $B_m \xrightarrow{s} B$ in $\mathcal{H}$ as $n,m\to\infty$. Let $u,u_k \in\mathcal{H}$ ($k=1,2,\dots$) and assume that $u_k\to u$ as $k\to\infty$. Then for every $\epsilon>0$ there exist $N,M,K$ such that for all $n>N,m>M,k>K$
	\be\label{unif-bound}
	\|A_nB_mu_k-ABu\|<\epsilon,
	\en
where $\|\cdot\|$ is the norm on $\mathcal{H}$.
\end{lemma}

\begin{proof}
This lemma is close to \cite[III Lemma 3.8]{kato}, but we write the proof in full detail. We rewrite $A_nB_mu_k-ABu$ as
	$$
	A_nB_mu_k-ABu
	=
	A_nB_m(u_k-u)+A_n\left(B_m-B\right)u+\left(A_n-A\right)Bu.
	$$
We claim that all three terms tend to $0$; the first two due to the uniform boundedness principle (UBP) \cite[\S2.2]{brezis}. Consider the first term. Letting $\mathcal{L}(\mathcal{H})$ be the space of bounded linear operators on $\mathcal{H}$, and letting $\|\mathcal{F}\|_{\mathcal{L}(\mathcal{H})}$ be the operator norm of the operator $\mathcal{F}\in\mathcal{L}(\mathcal{H})$, we write
	$$
	\|A_nB_m(u_k-u)\|
	\leq
	\left(\sup_n\|A_n\|_{\mathcal{L}(\mathcal{H})}\right)\left(\sup_m\|B_m\|_{\mathcal{L}(\mathcal{H})}\right)\|u_k-u\|.
	$$
By the UBP, there exists some $C$ such that
	$$
	\left(\sup_n\|A_n\|_{\mathcal{L}(\mathcal{H})}\right)\left(\sup_m\|B_m\|_{\mathcal{L}(\mathcal{H})}\right)<C,
	$$
so we simply need to pick $K$ large enough so that for any $k>K$ it holds that $\|u_k-u\|<\epsilon/3C$. We can bound the second term by $\epsilon/3$ using the UBP and using the fact that $B_m\xrightarrow{s}B$ in $\mathcal{H}$. We can bound the last term by $\epsilon/3$ using the fact that $A_n\xrightarrow{s}A$ in $\mathcal{H}$. Combining these results, we find that \eqref{unif-bound} indeed holds.
\end{proof}

We let $H^k=H^k(\T^d)$ be the usual Sobolev space of order $k$, where $\T^d$ is the $d$-dimensional torus. The inner product and the norm in $H^k$, respectively,  are denoted by $\left<\cdot,\cdot\right>_k$ and $\|\cdot\|_k$. We denote
	\be
	\mathcal{L}(k,l)=\left\{L:H^k\to H^l\;|\;L\text{ is linear and bounded}\right\}
	\en
with the operator norm $\|\cdot\|_{k,l}$. Now, consider a family of operators $\left\{K^\l\right\}_{\l\in[0,1)}$ that satisfy
	\begin{subequations}
		\begin{equation}\tag{H1}\label{H1}
		K^\l\in\mathcal{L}(0,0)\text{ are selfadjoint, and }\sup_\l\|K^\l\|_{0,0}<C
		\end{equation}
		\begin{equation}\tag{H2}\label{H2}
		\text{For any }\sigma\in[0,1),\;K^\l\xrightarrow{s} K^\sigma\text{ as }\l\to\sigma.
		\end{equation}
	\end{subequations}
Define
	\be\label{schrodinger-op}
	A^\l:=-\triangle+K^\l,\hspace{.5cm}\l\in[0,1),
	\en
which is a family of selfadjoint operators in $H^0$ with domain $H^2$ (see \cite[V-\S4.1]{kato}) and with pure point spectrum tending to $+\infty$ (\cite[IV-Theorem 5.35]{kato}).

For brevity, we shall denote $K=K^0,A=A^0$.

\begin{lemma}
The operators $K^\l$ are relatively compact perturbations of $\triangle$.
\end{lemma}

\begin{proof}
We first note that ``relatively compact with respect to $\triangle$'' is the same as being ``compact as an operator from $H^2$ to $H^0$.'' Now, let $I:H^2\hookrightarrow H^0$ be the inclusion operator, which is compact by the Rellich theorem. Denoting $\bar{K}^\l=K^\l\large|_{H^2}$, the restriction of the operator $K^\l$ to $H^2$, then $\bar{K}^\l=K^\l\circ I$ is a composition of bounded operators. Since a composition of bounded operators is compact if one of them is compact, we conclude that $\bar{K}^\l$ is a compact operator from $H^2$ to $H^0$.
\end{proof}

Let
	\begin{align}
	P_n:H^0\to H^0=&\text{ the projection operator onto the eigenspace associated with}\\
	&\text{ the first $n$ eigenvalues (counting multiplicity) of $A^0$.}\nonumber
	\end{align}
Note that, in fact, $R(P_n)\subseteq H^2$, and, therefore, one can speak of $P_n$ as an operator $H^0\to H^0$ or $H^0\to H^2$ with little difference. In what follows, the domain and range of the operators $P_n$ will be understood by context.

Clearly, $\|P_n\|_{0,0}=1$ for all $n$. Moreover, for any $u\in H^0$,
	\be
	\|P_nu-u\|_0\to0
	\en
as $n\to\infty$ due to the definition of $P_n$.

\begin{lemma}\label{lemma-restrict}
For any $u\in H^k,\;k=1,2$,
	\be\label{eq:proj-conv}
	\|P_nu-u\|_k\to0
	\en
as $n\to\infty$.
\end{lemma}

\begin{proof}
We first consider the case $k=2$. Let $(\alpha_i,v_i)$ be eigenvalue-eigenvetor pairs of $A$ in ascending order (counting multiplicity). Let $u\in H^2$ and suppose that $u=\sum_{j=1}^\infty a_jv_j$ in $H^0$. We wish to show that $\lim_{n\to\infty}\|\sum_{j=n}^\infty a_jv_j\|_2\to0$. We write, with some Poincar\'{e} constant $\kappa$,
	$$
	\left\|\sum\nolimits_{j=n}^\infty a_jv_j\right\|_2
	\leq
	\kappa\left\|\triangle\sum\nolimits_{j=n}^\infty a_jv_j\right\|_0
	+
	\left\|\sum\nolimits_{j=n}^\infty a_jv_j\right\|_0.
	$$
The second term tends to $0$ since $\|P_nu-u\|_0\to0$, and the first term can be bounded by
	$$
	\kappa\left\|K\sum\nolimits_{j=n}^\infty a_jv_j\right\|_0+\kappa\left\|A\sum\nolimits_{j=n}^\infty a_jv_j\right\|_0.
	$$
The first of these terms is bounded by $C\kappa\left\|\sum\nolimits_{j=n}^\infty a_jv_j\right\|_0$ which again tends to $0$. Considering the square of the second term, we have
	$$
	\kappa^2\left\|A\sum\nolimits_{j=n}^\infty a_jv_j\right\|_0^2
	=
	\kappa^2\left\|\sum\nolimits_{j=n}^\infty a_j\alpha_jv_j\right\|_0^2
	=
	\kappa^2\sum\nolimits_{j=n}^\infty\left|a_j\alpha_j\right|^2
	\to0
	$$
as $n\to\infty$ since $Au\in H^0$. Therefore we deduce that $\|P_nu-u\|_2\to0$. Since we also know that for any $v\in H^0$, $\|P_nv-v\|_0\to0$, we can conclude that for any $w\in H^1$, $\|P_nw-w\|_1\to0$ by interpolation, by using the Fourier representation of Sobolev spaces.
\end{proof}

For our next lemma, we recall the definition of $H^{-1}\supseteq H^0$ as the dual space of $H^1$ 
	$$
	H^{-1}=(H^1)^*=\{g\;|\;g\text{ is a bounded linear functional on }H^1\}
	$$
via the scalar product $\left<f,g\right>_0$ in $H^0$. Suppose that $f\in H^1$. Then
	\be
	\left|\left<f,g\right>_0\right|\leq\|f\|_1\|g\|_0
	\en
and therefore one may think of $g$ as being an element of $H^{-1}$ with
	\be
	\|g\|_{-1}=\inf_{C>0}\left\{\left|\left<f,g\right>_0\right|\leq C\|f\|_1\right\}\leq\|g\|_0.
	\en

\begin{lemma}\label{h-1}
One can extend the definition of $P_n$ to map $H^{-1}\to H^{-1}$, the extensions $\bar{P}_n$ remain uniformly bounded, and for any $v\in H^{-1}$, $\|\bar{P}_nv-v\|_{-1}\to0$ as $n\to\infty$.
\end{lemma}

\begin{proof}
By Lemma \ref{lemma-restrict}, there exists some $C$ such that $\sup_n\|P_n\|_{1,1}<C$. Consider the adjoint operator to $P_n:H^1\to H^1$ denoted by $\bar{P}_n=(P_n)^*:H^{-1}\to H^{-1}$. It is continuous, and $\|\bar{P}_n\|_{-1,-1}=\|P_n\|_{1,1}$ (see \cite[III-\S3.3]{kato}). Hence $\sup_n\|\bar{P}_n\|_{-1,-1}<C$.

Let us try to get a better understanding of the action of $\bar{P}_n$ on $H^{-1}$. Consider the $\left<H^1,H^{-1}\right>$ pairing $\left<f,\bar{P}_ng\right>$ with $f\in H^1$ and $g\in H^{0}$ thought of as an element of $H^{-1}$:
	\be
	\left<f,\bar{P}_ng\right>=\left<P_nf,g\right>=\left<P_n f,g\right>_0=\left<f,P_n g\right>_0.
	\en
Hence, on the subspace $H^0\subseteq H^{-1}$ it holds that $\bar{P}_n=P_n$, and, therefore, $\bar{P}_n$ is an extension of $P_n$. Letting $g\in H^0$, we have $\|\bar{P}_ng-g\|_{-1}\leq\|P_ng-g\|_0\to0$. Since $H^0$ is dense in $H^{-1}$, we may replace $g\in H^0$ by any element $v\in H^{-1}$ and obtain $\|\bar{P}_nv-v\|_{-1}\to0$.
\end{proof}

Since we have established that $\bar{P}_n$ and $P_n$ have similar behavior, in what follows we shall refer only to $P_n$, and the precise domain and range will be understood from context.

Define the \emph{truncated operator}	
	\be
	A^\l_n:=P_nA^\l P_n
	\en
to be the spectral projection in $H^0$ of the operator $A^\l$ onto the eigenspace associated with the first $n$ eigenvalues (counting multiplicity) of $A^0$. Thus, $A_n^\l$ depends upon two parameters. We are interested in its behavior as $n\to\infty$ and $\l\to0$.

For brevity, we shall write $A_n=A_n^0$.

%

\begin{prop}[Resolvent set]\label{uniform}
Let $\sigma\in\rho(A)$. Then there exist $N=N(\sigma)>0$ and $\l^*=\l^*(\sigma)>0$ such that $\sigma\in\rho(A_n^\la)$ for all $n>N$ and for all $0<\l<\l^*$.
\end{prop}

\begin{proof}
We prove it by contradiction. If the claim were not true, then for any $k$ there would exist $n_k>k$ and $0<\l_k<\frac{1}{k}$ such that $\sigma\in\Sigma(A_{n_k}^{\l_k})$. Hence, if we abuse notation and eliminate the $k$'s, there exist $0\neq u_n\in D$ such that
	\be\label{eig-eq}
	A_n^{\l_n} u_n=\sigma u_n
	\en
where $\l_n\to0$ as $n\to\infty$. Normalize $\|u_n\|_{0}=1$ and multiply the equation \eqref{eig-eq} by $u_n$ (taking $H^0$ inner product) to get
	\be
	\left<A_n^{\l_n} u_n,u_n\right>_{0}=\sigma,
	\en
which we rewrite as
	\be\label{h1-bound}
	\|\nabla P_nu_n\|^2_0=\sigma-\left<K^{\l_n}P_nu_n,P_nu_n\right>_0.
	\en
We readily estimate $\left|\left<K^{\l_n}P_nu_n,P_nu_n\right>_0\right|\leq\|K^{\l_n}\|_{0,0}\|u_n\|_0^2\leq C$, hence $\|P_nu_n\|_{1}\leq \sigma+C<\infty$. Therefore there exists a subsequence of $P_nu_n$ that converges in $L^2$ to some $u\neq0$ (for simplicity we do not change the index):
	\be\label{converge1}
	\lim_{n\to\infty}\|P_nu_n-u\|_{0}=0.
	\en
In addition, since $P_nu_n$ is bounded in $H^1$ it has a weakly convergent subsequence in $H^1$ which necessarily converges to $u$: $P_nu_n\xrightarrow{w}u$ in $H^1$ (we keep the same index $n$), hence $u\in H^1$. Therefore
	\be\label{eq:weak-h-1}
	\triangle  P_nu_n\xrightarrow{w}\triangle u
	\en
in $H^{-1}$. By Lemma \ref{h-1}, $P_n\xrightarrow{s}I$ in $H^{-1}$, and one \emph{would like} to conclude that $P_n\triangle P_nu_n\xrightarrow{w}\triangle u$ in $H^{-1}$. While the conclusion is true, showing it is somewhat more delicate: considering the $\left<H^1,H^{-1}\right>$ pairing for some $v\in H^1$ we have
	\begin{align*}
	\left<v,P_n\triangle P_nu_n-P_n\triangle u\right>
	&=
	\left<P_nv,\triangle P_nu_n-\triangle u\right>\\
	&=
	\left<v,\triangle P_nu_n-\triangle u\right>+\left<(P_n-I)v,\triangle P_nu_n-\triangle u\right>\to0
	\end{align*}
due to \eqref{eq:weak-h-1} for the first term, and since $P_n\xrightarrow{s} I$ in $H^1$ for the second term (recall that due to \eqref{eq:weak-h-1}, $\|\triangle P_nu_n-\triangle u\|_{-1}$ is uniformly bounded). We therefore have that $P_n\triangle P_nu_n-P_n\triangle u\xrightarrow{w}0$ in $H^{-1}$. Since $P_n\triangle u\xrightarrow{s}\triangle u$, we conclude that, indeed,
	\be
	P_n\triangle P_nu_n\xrightarrow{w}\triangle u
	\en
in $H^{-1}$.
We now consider \eqref{eig-eq} again:
	\be\label{eig-eq2}
	-P_n\triangle P_nu_n+P_nK^{\l_n}P_nu_n=\sigma u_n.
	\en
Recalling \eqref{H2} and using Lemma \ref{uniform-bound}, we have that $P_nK^{\l_n}P_nu_n\to Ku$ in $H^0$. In addition, we claim that $\sigma u_n\to \sigma u$ in $H^0$. Indeed, if $\sigma\neq0$, then by the structure of \eqref{eig-eq2} $P_nu_n=u_n$. If $\sigma=0$ the claim is trivially true. Therefore, letting $n\to\infty$, all three terms above converge weakly in the $H^{-1}$ sense to
	\be
	-\triangle u+Ku=\sigma u.
	\en
But since $u$ and $Ku$ are elements of $H^0$, one can bootstrap (by elliptic regularity) and conclude that $0\neq u\in H^2$, which contradicts the fact that $\sigma\in\rho(A)$.
\end{proof}

\begin{cor}\label{cor:limit}
If there exist sequences $n_k\to\infty$, and $\l_k\to0$ such that $\sigma\in\Sigma(A_{n_k}^{\l_k})$ for all $k$, then $\sigma\in\Sigma(A)$.
\end{cor}

\begin{prop}\label{first-prop}
Suppose that $0\in\rho(A)$. Let $\l^*=\l^*(0)$ and $N=N(0)$ be as in Proposition \ref{uniform}. Then for all $0<\l<\l^*$ and for all $n>N$, $neg(A_n^\l)=neg(A_n)$.
\end{prop}

\begin{proof}
We first remark that since $P_n$ is the orthogonal projection onto an eigenspace associated to eigenvalues of the operator $A$, it holds that $\Sigma(A_n)\subseteq\Sigma(A)$. Since $A$ has only finitely many negative eigenvalues, and since $P_n$ is the projection onto the first $n$ eigenvalues of $A$, it also holds that $neg(A)=neg(A_n)$ for all $n$ sufficiently large. The benefit of using the projection operators $P_n$ is that now we only need to deal with the family of operators $A_n^\l$ acting on finite-dimensional subspaces of $H^2$. The family of operators $A_n^\l$ is continuous in $\l$ for $0\leq\l<1$. Here we used the simple fact that all norms on $\R^N$ are equivalent, combined with the hypothesis \eqref{H2}. Therefore the set of eigenvalues $\Sigma(A_n^\l)$ varies continuously in $\l$, for all $0\leq\l<1$ and fixed $n$.

To prove the proposition we again argue by contradiction: If the assertion were not true, then there would exist a sequence $\l_{n_k}\to0$ and a sequence $n_k\to\infty$ such that $neg(A_{n_k}^{\l_{n_k}})\neq neg(A_{n_k})$. We drop the $k$ subscript for notational simplicity. Fix $n$. Since $\Sigma(A_n^\l)$ varies continuously in $\l$ for all $0\leq\l<1$, $neg(A_n^\l)$ can have jumps only for $\l_n'$ for which $0\in\Sigma(A_n^{\l_n'})$. Therefore, still fixing $n$, the inequality $neg(A_n^{\l_n})\neq neg(A_n)$ implies that there must exist $0<\l_n'<\l_n$ for which $0\in\Sigma(A_n^{\l_n'})$. By Corollary \ref{cor:limit}, $0\in\Sigma(A)$, which is a contradiction.
\end{proof}

\begin{prop}\label{second-prop}
Suppose that $0\in\Sigma(A)$. There exist $\l^*$ and $N$ such that for all $0<\l<\l^*$ and for all $n>N$, $neg(A_n^\l)\geq neg(A_n)$.
\end{prop}

\begin{proof}
The proof of this assertion is very similar in nature to the proof of Proposition \ref{first-prop}. The claim is trivial if $neg(A)=0$. Therefore we assume that $neg(A)\geq1$. Let $\sigma\in\rho(A)$ satisfy
\begin{enumerate}
\item
$\sigma<0$
\item
$[\sigma,0)\cap\Sigma(A)=\emptyset$
\end{enumerate}
Then there exist $\l^*=\l^*(\sigma)$ and $N=N(\sigma)$ as in Proposition \ref{uniform} such that $\sigma\in\rho(A_n^\l)$ for all $n>N$ and $0<\l<\l^*$. If the assertion of this proposition were not true, then there would exist a sequence $\l_n\to0$ and $n\to\infty$ such that $neg(A_n^{\l_n})< neg(A_n)$. That is, for fixed $n$, an eigenvalue of $\Mnl$ must cross $0$ from left to right as $\l$ varies from 0 to $\l_n$, due to the continuous dependence of $\Sigma(A_n^\l)$ on $\l$ (remember that the spectrum is real!). In particular, there would have to exist some $0<\l_n'<\l_n$ for which $\sigma\in\Sigma(A_n^{\l_n'})$. Once again, by Corollary \ref{cor:limit} this implies that $\sigma\in\Sigma(A)$, which is a contradiction. This proves that $neg(A_n^\l)\geq neg(A_n)$.
\end{proof}

In fact, we will now discuss how the results of this section can be modified to handle operators of the form
	\be
	M^\l
	=
	\mattwo{A_1^\l}{K^\l}{(K^\l)^*}{-A_2^\l}
	\en
with $A_i^\l=-\triangle+K_i^\l$, and $K^\l,K_1^\l,K_2^\l$ endowed with the same hypotheses \eqref{H1} and \eqref{H2}. In this case, one has to define two families of projection operators, $P_n$ and $Q_n$ (with similar properties as those of $P_n$ discussed above), associated to $A_1$ and $A_2$ respectively, to form:
	\be
	M^\l_n
	=
	\mattwo{P_nA_1^\l P_n}{P_nK^\l Q_n}{Q_n(K^\l)^*P_n}{-Q_nA_2^\l Q_n}.
	\en

Let us show, for instance, how Proposition \ref{uniform} may be restated:

\begin{prop}\label{uniform2}
Let $\sigma\in\rho(M^0)$. Then there exist $N=N(\sigma)>0$ and $\l^*=\l^*(\sigma)>0$ such that $\sigma\in\rho(M_n^\l)$ for all $n>N$ and for all $0<\l<\l^*$.
\end{prop}

\begin{proof}
We mimic the proof by contradiction presented in Proposition \ref{uniform}, by letting
	$$
	u_n=\left(\begin{array}{c}v_n\\w_n\end{array}\right)
	$$
be a sequence of unit vectors $\|u_n\|_0^2=\|v_n\|_0^2+\|w_n\|_0^2=1$ with the eigenvalue $\sigma$:
	\be\label{contra-arg}
	M_n^{\l_n} u_n=\sigma u_n.
	\en
Consider the first equation
	\be
	P_nA_1^{\l_n}P_nv_n+P_nK^{\l_n}Q_nw_n=\sigma v_n
	\en
and multiply it by $v_n$ (taking $H^0$ inner product) to get
	\be
	\left<A_1^{\l_n}P_nv_n,P_nv_n\right>_0+\left<K^{\l_n}Q_nw_n,P_nv_n\right>_0=\sigma\|v_n\|_0^2,
	\en
which gives us an expression analogous to \eqref{h1-bound}:
	\be
	\|\nabla P_nv_n\|_0^2
	=
	\sigma\|v_n\|_0^2-\left<K^{\l_n}Q_nw_n,P_nv_n\right>_0-\left<K_1^{\l_n}P_nv_n,P_nv_n\right>_0.
	\en
As was the case in \eqref{h1-bound}, the right hand side is uniformly bounded, and therefore we can deduce $H^0$ convergence of $P_nv_n$ to some $v$ (possibly $0$), and weak $H^{-1}$ convergence of $P_n\triangle P_nv_n$ to $\triangle v$.

Repeating this argument for the second equation we get $H^0$ convergence of $Q_nw_n$ to some $w$ (possibly $0$) and weak $H^{-1}$ convergence of $Q_n\triangle Q_nw_n$ to $\triangle w$. Moreover, $u^T=(v,w)\neq(0,0)$ due to the normalization of $u_n$. We now rewrite \eqref{contra-arg}:
	\begin{align*}
	-P_n\triangle P_nv_n+P_nK_1^{\l_n}P_nv_n+P_nK^{\l_n}Q_nw_n
	&=
	\sigma v_n\\
	Q_n\left(K^{\l_n}\right)^*P_nv_n+Q_n\triangle Q_nw_n-Q_nK_2^{\l_n}Q_nw_n
	&=
	\sigma w_n.
	\end{align*}
As $n\to\infty$, both equations converge in the $H^{-1}$ sense to
	\begin{align*}
	-\triangle v+K^0_1v+K^0w
	&=
	\sigma v\\
	\left(K^0\right)^*v+\triangle w-K^0_2w
	&=
	\sigma w.
	\end{align*}
However, all terms without ``$\triangle$'' are elements of $H^0$, and therefore we may bootstrap $v$ and $w$ (now elements of $H^2$ by elliptic regularity), and conclude that these equations hold in the strong sense, i.e., that
	$$
	M^0 u=\sigma u,
	$$
which is a contradiction to the assumption on $\sigma$.
\end{proof}

\begin{cor}\label{cor:limit2}
If there exist sequences $n_k\to\infty$, and $\l_k\to0$ such that $\sigma\in\Sigma(M_{n_k}^{\l_k})$ for all $k$, then $\sigma\in\Sigma(M^0)$.
\end{cor}

Considering the operator $M$ rather than $A$, the statements of Propositions \ref{first-prop} and \ref{second-prop} become:

\begin{prop}\label{first-prop2}
Suppose that $0\in\rho(M^0)$. Let $\l^*=\l^*(0)$ and $N=N(0)$ be as in Proposition \ref{uniform2}. Then for all $0<\l<\l^*$ and for all $n>N$, $neg(M_n^\l)=neg(M_n)$.
\end{prop}

\begin{proof}
The idea is the same as in Proposition \ref{first-prop}: assuming the statement to be false, we obtain a sequence $\l_n\to0$ for which $neg(M_n^{\l_n})\neq neg(M_n)$. However, by the continuous dependence of the spectrum of $M_n^\l$ on $\l$ (for \emph{fixed} $n$), some eigenvalue of $M_n^\l$ must cross $0$ as $\l$ varies from $0$ to $\l_n$, say $\l_n'$. Since $\l_n\to0$, it also holds that $\l_n'\to0$. Therefore $0\in\Sigma(M_n^{\l_n'})$ and by Corollary \ref{cor:limit2}, $0\in\Sigma(M^0)$ which is a contradiction.
\end{proof}

\begin{prop}\label{second-prop2}
Suppose that $0\in\Sigma(M^0)$. There exist $\l^*$ and $N$ such that for all $0<\l<\l^*$ and for all $n>N$, $neg(M_n^\l)\geq neg(M_n)$.
\end{prop}

\begin{proof}
As in Proposition \ref{second-prop}, let $\sigma\in\rho(M^0)$ satisfy
\begin{enumerate}
\item
$\sigma<0$
\item
$[\sigma,0)\cap\Sigma(M^0)=\emptyset$
\end{enumerate}
Now, as before, assume the assertion to be false. Then there exists a sequence $\l_n\to0$ for which $neg(M_n^{\l_n})< neg(M_n)$. Hence as $\l$ varies from $0$ to $\l_n$, an eigenvalue of $M_n^\l$ must cross $0$ from left to right ($n$ is fixed in this argument). In particular, due to our choice of $\sigma$, there must be a crossing of $\sigma$ from left to right as well, say at $\l_n'$. Then $\sigma\in\Sigma(M_n^{\l_n'})$ with $\l_n'\to0$. By Corollary \ref{cor:limit2}, $\sigma\in\Sigma(M^0)$ which is a contradiction.
\end{proof}

Finally, we note that Proposition \ref{uniform2} still holds with $\sigma\in\rho(M^{\l_0})$ for $\l_0>0$. In this case, the interval $(0,\l^*)$ must be replaced by some interval $(\l^*,\l^{**})$ that contains $\l_0$, but all other parts of the analysis still hold. We can formulate the contra-positive:
\begin{prop}\label{prop:limit}
If there exist sequences $n_k\to\infty$, and $\l_k\to\l_0$ such that $\sigma\in\Sigma(M_{n_k}^{\l_k})$ for all $k$, then $\sigma\in\Sigma(M^{\l_0})$.
\end{prop}

\section{Reformulation of the Problem}\label{setup-section}
In this section we introduce the function spaces that will be used throughout this paper, as well as set up notation, and transform our problem from a first-order non-selfadjoint problem, into a much simpler selfadjoint problem, involving only the spatial variable, and not the momentum. This problem, appearing in \eqref{problem}, will be the focus of the next sections of this work.

\subsection{The Function Spaces}
The function spaces we use throughout this paper are as follows: We denote by $\lpt$ and $H_P^k$ the usual Sobolev spaces of $P$-periodic functions on $\R$, that are square integrable on the interval $[0,P]$, as well as their first $k$ derivatives. In addition, we denote:

	\begin{eqnarray*}
	\lptz&=&\left\{h(x)\in \lpt \;\Bigg{|}\; \int_0^Ph\;dx=0\right\}\\
	H_{P,0}^k&=&\left\{h(x)\in H^k_P\;\Bigg{|}\; \int_0^Ph\;dx=0\right\}\\
	L_\pm^2&=&\left\{m(x,v)\;\Bigg{|}\;  m \text{ is } P\text{-periodic in }x, \left\|m\right\|_\pm^2:=\int_0^P\int_{\R^2}|m|^2w^\pm\;dv\;dx<\infty\right\}\\
	\end{eqnarray*}
where $w^\pm$ are the weights defined in \eqref{weight}. In addition, we use the following notation:
	
	\begin{align*}
	\|\cdot\|_{\lpt}\text{ and }\left<\cdot,\cdot\right>_{\lpt}  \text{ denote the norm and inner product in }L^2_{P}\text{ respectively}\\
	\|\cdot\|_{\pm}\text{ and }\left<\cdot,\cdot\right>_{\pm}  \text{ denote the norm and inner product in }L^2_{\pm}\text{ respectively}\\
	\end{align*}

\subsection{The Basic Equations}\label{equations1}

%
%

In the $1\frac{1}{2}$ dimensional case the RVM system becomes a system of scalar equations:
		\begin{subequations}\label{main-eq}
		\begin{align}
		&\p_tf^\pm+\hat{v}_1\p_xf^\pm\pm(E_1+\hat{v}_2B)\p_{v_1}f^\pm\pm(E_2-\hat{v}_1B)\p_{v_2}f^\pm
		=
		0						\label{vlasov}\\
		&\p_tE_1
		=
		-j_1						\label{ampere1}\\
		&\p_tE_2+\p_xB
		=
		-j_2						\label{ampere2}\\
		&\p_tB
		=
		-\p_xE_2					\label{faraday}\\
		&\p_xE_1
		=
		n_0+\rho						\label{gauss}
		\end{align}
		\end{subequations}
where $\hat{v}=v/\left<v\right>$ and $\left<v\right>=\sqrt{1+|v|^2}$, and
	\begin{subequations}
	\begin{align}
	\rho
	&=
	\int\left(f^+-f^-\right)\;dv	\\
	j_i
	&=
	\int\hat{v}_i\left(f^+-f^-\right)\;dv,\hspace{.25cm}i=1,2
	\end{align}
	\end{subequations}
and the external fields are replaced by the constant external radiation $n_0$. The linearized Vlasov equation is
	\begin{equation}\label{linvlasov}
	\left(\p_t+D^\pm\right)f^\pm
	=
	\mp\mu^\pm_e\hat{v}_1E_1\pm\mu^\pm_p\hat{v}_1B\mp\left(\mu^\pm_e\hat{v}_2+\mu^\pm_p\right)E_2,
	\end{equation}
where
	$$
	D^\pm
	=
	\hat{v}_1\p_x\pm\left(E_1^0+\hat{v}_2B^0\right)\p_{v_1}\mp\hat{v}_1B^0\p_{v_2}.
	$$

Since we are looking for a (purely) growing mode with exponent $\la$ (see \eqref{purely}), we replace everywhere $\p_t$ by $\la$. Thus, our equations for the $P$-periodic electric and magnetic (perturbed) potentials $\phi$ and $\psi$ become
	$$
	B=\p_x\psi
	$$
along with	
	$$
	E_2=-\la\psi,
	$$
which is a result of the integration of \eqref{faraday} and setting the constant of integration to be $0$, and, finally,
	$$
	E_1
	=
	-\p_x\phi-\la b,
	$$
where $b\in\R$ is meant to allow $E_1$ to have a nonzero mean.
%
Replacing $\p_t$ by $\la$, the linearized Vlasov equation becomes
	\begin{equation}\label{vlas}
	\left(\la+D^\pm\right)f^\pm
	=
	\pm\mu^\pm_e\hat{v}_1(\p_x\phi+\la b)\pm\mu^\pm_p\hat{v}_1\p_x\psi\pm\la\left(\mu^\pm_e\hat{v}_2+\mu^\pm_p\right)\psi,
	\end{equation}
along with Maxwell's equations
	\begin{subequations}\label{max-eq}
	\begin{align}
	&-\la\p_x\phi-\la^2b=\la E_1=-j_1\label{ampere1d1}\\
	&-\la^2\psi+\p_x^2\psi=\la E_2+\p_xB=-j_2\label{ampere1d2}\\
	&-\p_x^2\phi=\p_xE_1=\rho.\label{gauss1d}
	\end{align}
	\end{subequations}

We see that in these equations there is only dependence upon derivatives of the electric potential $\phi$, and never dependence upon $\phi$ itself. Therefore, throughout this paper we let $\phi\in\lptz$. Now we introduce the particle paths $\left(X^\pm(s;x,v),V^\pm(s;x,v)\right)$ of the equilibrium state, governed by the transport operators $D^\pm$, where $-\infty<s<\infty$. They satisfy the system of ordinary differential equations
	\begin{subequations}
	\begin{align}
	\dot{X}^\pm&=\hat{V}^\pm_1\\
	\dot{V}^\pm_1&=
	\pm E_1^0(X^\pm)\pm\hat{V}^\pm_2B^0(X^\pm)\\
	\dot{V}^\pm_2&=\mp\hat{V}^\pm_1B^0(X^\pm),
	\end{align}
	\end{subequations}
where $\dot{\square}=\p/\p s$ is the derivative along the characteristic curves, and the initial conditions are
	\begin{equation}
	\left(X^\pm(0,x,v),V^\pm(0,x,v)\right)
	=
	(x,v).
	\end{equation}
When there is no risk of confusion, we abbreviate $X^\pm=X^\pm(s)=X^\pm(s;x,v)$ and $V^\pm=V^\pm(s)=(V_1^\pm(s;x,v),V_2^\pm(s;x,v))$. Now we rewrite the Vlasov equation integrated along the particle paths. Here it is crucial that $e^\pm$ and $p^\pm$, and any function of these variables, are constant along the trajectories. This implies that, $\mu_e^\pm$ and $\mu_p^\pm$ are constants under $s$-differentiation. We multiply \eqref{vlas} by $e^{\la s}$ and notice that the left hand side becomes the perfect derivative $\frac{\p}{\p s}\left(e^{\la s}f\right)$. Integrating the right hand side along the particle paths, one has
	\begin{align*}
	&\pm \int_{-\infty}^0e^{\la s}\left(\mu^\pm_e\hat{V}^\pm_1(\p_x\phi(X^\pm)+\la b)+\mu^\pm_p\hat{V}^\pm_1\p_x\psi(X^\pm)+\la\left(\mu^\pm_e\hat{V}^\pm_2+\mu^\pm_p\right)\psi(X^\pm)\right)ds\\
	=&\pm \int_{-\infty}^0e^{\la s}\mu^\pm_e\left(\hat{V}^\pm_1\p_x\phi(X^\pm)+\la\phi(X^\pm)\right) ds\mp \int_{-\infty}^0\la e^{\la s}\mu^\pm_e\phi(X^\pm)\;ds\\
	&\pm \int_{-\infty}^0e^{\la s}\mu^\pm_p\left(\hat{V}^\pm_1\p_x\psi(X^\pm)+\la\psi(X^\pm)\right)ds\\
	&\pm \int_{-\infty}^0\la e^{\la s}\mu_e^\pm\left(\hat{V}^\pm_1b+\hat{V}^\pm_2\psi(X^\pm)\right)ds=I+II+III+IV.
	\end{align*}
Recalling that $D^+$ and $D^-$ both reduce to $\hat{v}_1\p_x$ when applied to functions of $x$ alone (and not $v$), we see that the integrands in terms $I$ and $III$ are $e^{\la s}\mu_e^\pm\left(D\phi+\la\phi\right)$ and $e^{\la s}\mu_e^\pm\left(D\psi+\la\psi\right)$, respectively, evaluated along the appropriate particle paths. Therefore, both integrands become no more than $\frac{d}{ds}\left(e^{\la s}\mu_e^\pm\phi\right)$ and $\frac{d}{ds}\left(e^{\la s}\mu_e^\pm\psi\right)$. We conclude that the terms $I$ and $III$ become $\pm e^{\la s}\mu_e^\pm\phi(x)$ and $\pm e^{\la s}\mu_p^\pm\psi(x)$, with no boundary terms due to our decay assumptions. The other terms are kept in integral form as above. Since $\mu_e^\pm$ are constant along the trajectories, we may evaluate them at $s=0$, and they have no role in the integration. After dividing both sides by the exponent, we finally have
	\begin{align}\label{fpm}
	f^\pm(x,v)=&\pm\mu^\pm_e\phi(x)\pm\mu^\pm_p\psi(x)\\
	&\mp\mu^\pm_e\int_{-\infty}^0\la e^{\la s}\left[\phi(X^\pm(s))-\hat{V}^\pm_2(s)\psi(X^\pm(s))-b\hat{V}^\pm_1(s)\right]\;ds.\nonumber
	\end{align}

We simplify this expression by introducing the operators $\Ql_\pm: L^2_\pm\to L^2_\pm$, defined by:
	$$\left(\Ql_\pm k\right)(x,v)=\int_{-\infty}^0\la e^{\la s}k\left(X^\pm(s;x,v),V^\pm(s;x,v)\right)\;ds$$
where $k=k(x,v):[0,P]\times\R^2\to\R$.

\begin{rek}\label{ql-domain}
We note that if $h(x)\in\lpt$, then $\tilde{h}(x,v):=h(x), \;(x,v)\in[0,P]\times\R^2$, is clearly in $L^2_\pm$. Therefore, $\Ql_\pm$ act on functions in $\lpt$ as well. 
\end{rek}

With our definition of $\Ql_\pm$, \eqref{fpm} becomes

	\begin{equation}\label{partdist1d}
	\begin{gathered}
	f^+(x,v)
	=
	+\mu^+_e\phi(x)+\mu^+_p\psi(x)-\mu^+_e\Ql_+\phi+\mu^+_e\Ql_+(\hat{v}_2\psi)+b\mu^+_e\Ql_+\hat{v}_1,\\
	f^-(x,v)
	=
	-\mu^-_e\phi(x)-\mu^-_p\psi(x)+\mu^-_e\Ql_-\phi-\mu^-_e\Ql_-(\hat{v}_2\psi)-b\mu^-_e\Ql_-\hat{v}_1.
	\end{gathered}
	\end{equation}

%


\begin{lemma}[Properties of $D^\pm$]
$D^\pm$ are skew-adjoint operators on $\lwt$. Their null spaces $\ker{D^\pm}$ consist of all functions $g=g(x,v)$ in $\lwt$ that are constant on each connected component in $\R\times\R^2$ of $\{e^\pm=const\text{ and }p^\pm=const\}$. In particular, $\ker{D^\pm}$ contain all functions of $e^\pm$ and of $p^\pm$.
\end{lemma}

\begin{proof}
We show for the `$+$' case, and drop the $+$ superscripts. It is straightforward to verify that $De=Dp=0$. Therefore $\ker{D}$ contains all functions of $e$ and of $p$. Skew-adjointness is easily seen due to integration by parts, as $D$ is a first-order differential operator. Derivatives that ``hit" $w$ vanish, since $w=w(e)$ is a function of $e$.
\end{proof}

\begin{defi}\label{proj-define}
We define $\Proj^\pm$ to be the orthogonal projection operators of $L_\pm^2$ onto $\ker D^\pm$.
\end{defi}

\begin{lemma}\label{proj-parity}
The projection operators $\Proj^\pm$ preserve parity with respect to the variable $v_1$.
\end{lemma}

\begin{proof}
Let us demonstrate for $\Proj^+$ and drop the $+$ superscript to simplify notation. The demonstration for $\Proj^-$ is identical. Recall that
	$$
	D=D^+=\hat{v}_1\p_x+
	(E_1^0+\hat{v}_2B^0)\p_{v_1}-\hat{v}_1B^0\p_{v_2}.
	$$
Now, let $f=f(x,v_1,v_2)$ and let $R$ be the operator that reverses $v_1$: $Rh(x,v_1,v_2)=h(x,-v_1,v_2)$. Then
	$$
	D(Rf)
	=
	-R(\hat{v}_1\p_xf)-R(
	(E_1^0+\hat{v}_2B^0)\p_{v_1}f)+R(\hat{v}_1B^0\p_{v_2}f)
	=
	-R(Df)
	$$
Therefore $f\in\ker D$ if and only if $Rf\in\ker D$. This implies that one can find a basis of even and odd functions (in the variable $v_1$) to the space $\ker D$. To show that if $f$ is even (odd) in $v_1$ then $\Proj f$ is also even (odd) in $v_1$, we let $g\in \ker D$ be, without loss of generality, even or odd. Then it must hold that $\iint(f-\Proj f)g\;w\;dx\;dv=0$. In the case that $f$ is even, we change variables $v_1\to-v_1$ to get $\pm\iint(f-R(\Proj f))g\;w\;dx\;dv=0$ and, therefore, $R(\Proj f)=\Proj f$. Here the $\pm$ depends on the parity of $g$. The odd case is treated in the same way.
\end{proof}

We now state the important properties of the operators $\Ql_\pm$. These properties are discussed in \cite[Lemma 2.5]{benartzi1}, with a slightly different weight: in \cite{benartzi1} we had $e^+=e^-=\left<v\right>$ due to the symmetry assumptions, and therefore $w^+=w^-$ were simply denoted by $w$. This difference, though subtle, does not impact the proof in a significant way, and we therefore do not include the proof here.
	\begin{lemma}[Properties of $\Ql_\pm$]\label{propql}
	Let $0<\la<\infty$.
		\begin{enumerate}
		\item\label{qlnorm}

		$\Ql_\pm$ map $\lwt\to\lwt$ with operator norm = 1. Moreover, recalling Remark \ref{ql-domain}, $\Ql_\pm$ are also bounded as operators $\lpt\to\lwt$.

		\item\label{laz}
	For all $m(x,v)\in\lwt$, $\left\|\Ql_\pm m-\Proj^\pm m\right\|_\pm\to0$ as $\la\to0$.

		\item\label{lanotz}
		If $\sigma>0$, then $\left\|\Ql_\pm-\Qs_\pm\right\|=O(|\la-\sigma|)$ as $\la\to\sigma$, where $\|\cdot\|$ is the operator norm from $\lwt$ to $\lwt$.

		\item\label{lainfty}
		For all $m(x,v)\in\lwt$, $\left\|\Ql_\pm m-m\right\|_\pm\to0$ as $\la\to\infty$.
		
		\end{enumerate}
	\end{lemma}

\subsection{The Operators}
In addition to the definitions of $\Aonez, \Atwoz$ and $\Bz$ in \S\ref{mainres}, we define the following operators depending on a real parameter $0<\lambda<\infty$, acting on $\lptz,\lpt,\lpt$, with domains $\hptz,\hpt,\lpt$, respectively:
	\begin{eqnarray*}
	\Aonel h&=&-\p_x^2h-\left(\sum_\pm\int\mu^\pm_e\;dv\right)h+\sum_\pm\int\mu^\pm_e\Ql_\pm h\;dv,\\
	\Atwol h&=&-\p_x^2h+\la^2h-\left(\sum_\pm\int\hat{v}_2\mu^\pm_p\;dv\right)h-\sum_\pm\int\mu^\pm_e\hat{v}_2\Ql_\pm (\hat{v}_2h)\;dv,\\
	\Bl h&=&\left(\sum_\pm\int\mu^\pm_p\;dv\right)h+\sum_\pm\int\mu^\pm_e\Ql_\pm(\hat{v}_2h)\;dv.
	\end{eqnarray*}
As we have seen in Lemma \ref{propql}, $\Ql_{\pm}\to\Proj^\pm$ strongly as $\la\to0$; this hints at the relationship between the operators defined here (with $\l$ superscript) and the operators defined in \S\ref{mainres} (with a $0$ superscript). The rigorous results may be found in \S\ref{the-operators}. We also define the following multiplication operators with domain $\R$ and range $\lpt$, depending on the parameter $0<\la<\infty$:
	\begin{eqnarray*}
	\Cl(b)&=&b\sum_\pm\int\mu^\pm_e\Ql_\pm\left(\hat{v}_1\right)dv,\\
	\Dl(b)&=&b\sum_\pm\int\hat{v}_2\mu^\pm_e\Ql_\pm\left(\hat{v}_1\right)dv,\\
	\end{eqnarray*}
and a constant depending on $\la$:
	$$
	l^\la=\frac{1}{P}\sum_\pm\int_0^P\int\hat{v}_1\mu^\pm_e\Ql_\pm\left(\hat{v}_1\right)dv\;dx,
	$$
which is closely related to the number $l^0$ defined in \S\ref{mainres}. Now let us write the formulas for the adjoint operators of $\Bl$, $\Cl$ and $\Dl$. The derivation of these expressions appears in \cite{benartzi1}:
%
%
	\begin{align*}
	\left(\Bl\right)^*k&=\left(\sum_\pm\int\mu^\pm_p\;dv\right)k+\sum_\pm\int\mu^\pm_e\hat{v}_2\Ql_\pm k\;dv,\\
%
	\left(\Cl\right)^*k&=\sum_\pm\int_0^P\int\mu^\pm_e\Ql_\pm\left(\hat{v}_1\right)k(x)\;dv\;dx,\\
	\left(\Dl\right)^*k&=\sum_\pm\int_0^P\int\hat{v}_2\mu^\pm_e\Ql_\pm\left(\hat{v}_1\right)k(x)\;dv\;dx.
	\end{align*}

\subsection{The Matrix Operator $\Ml$}\label{ml}
As in \cite{benartzi1}, after plugging in the expressions for $f^+$ and $f^-$ into Maxwell's Equations, one finds three equations for the three potentials $\phi,\psi$ and $b$, all depending upon the parameter $\l$. We write the three equations in matrix form as
	\be\label{problem}
	\Ml u=0,
	\en
where the matrix operator $\Ml:\lpt\times\lpt\times\R\to \lpt \times \lpt \times \R$ is defined as
	\be\label{main-mat-op}
	\Ml
	=
	\matthree{-\Aonel}{\Bl}{\Cl}{\left(\Bl\right)^*}{\Atwol}{-\Dl}{\left(\Cl\right)^*}{-\left(\Dl\right)^*}{-P\left(\la^2-l^\la\right)}.
	\en
with domain $\hpt\times\hpt\times\R$.
Formally, to prove our main theorem, it suffices to show that $\Ml$ has a nontrivial kernel for some $0<\la<\infty$.
Finally, we define
	\be\label{mz}
	\Mz
	=
	\matthree{-\Aonez}{\Bz}{0}{\left(\Bz\right)^*}{\Atwoz}{0}{0}{0}{Pl^0}.
	\en

\begin{rek}\label{remark}
As mentioned before, since $\phi$ only matters up to a constant, we restrict the domain of $\Ml$ and $\Mz$ to $\hptz\times\hpt\times\R$. Indeed, making this restriction is important. Letting $(\phi,\psi,b)=u_{triv}^T=(1,0,0)$ we notice that $\Ml u_{triv}=0$ for any $\la\geq0$. However, $u_{triv}$ is a trivial solution that is of no interest for us, since it would generate a trivial solution $(f,E,B)=(0,0,0)$. Moreover, multiples of $u_{triv}$ are the only trivial solutions. Indeed, whenever either $\psi$ or $b$ are nonzero, the linearized equations become nontrivial.
\end{rek}

\section{Analysis for Small and Large Values of $\l$}\label{large-small}
Our goal is to find for $\Ml$ a nontrivial kernel for some $0<\l<\infty$. Our method is to ``compare'' its spectrum when $\l=0$ and when $\l\to\infty$ and detect some kind of discrepancy that would indicate that an eigenvalue crossed through $0$. To make this intuitive process rigorous, we shall first consider finite-dimensional approximations of the operators $\Ml$, study them, and then retrieve the original operators $\Ml$.

First, let us make a general comment on `diagonalization' of real block matrices. Consider a symmetric block matrix
	\be
	M=\matthree{A_1}{B}{C}
	{B^*}{A_2}{D}
	{C^*}{D^*}{A_3}
	\en
mapping $\R^n\times\R^n\times\R^n\to\R^n\times\R^n\times\R^n$ and the matrix
	\be
	J=\matthree{I_n}{0}{0}{J_1}{I_n}{0}{J_2}{J_3}{I_n}
	\en
with $J_i,\;i=1,2,3$, to be determined later. Our goal is to `diagonalize' $M$ by conjugating with $J$, to get
	\be
	\Delta=J^*MJ=\matthree{\Delta_1}{0}{0}
	{0}{\Delta_2}{0}
	{0}{0}{\Delta_3}.
	\en

While $M$ and $\Delta$ do not share the same eigenvalues and eigenvectors, they both do have negative and positive eigenspaces (that is, eigenspaces associated to negative, resp. positive, eigenvalues) of the same dimension:
	\be
	neg(M)=neg(\Delta)\hspace{1cm} pos(M)=pos(\Delta).
	\en
Indeed, suppose that $u$ is in the negative eigenspace of $J^*MJ$: $u^TJ^*MJu<0$. Then $(Ju)^TMJu<0$ hence $v:=Ju$ is in the negative eigenspace of $M$. Since $J$ has a trivial kernel, this implies that $neg(M)\geq neg(J^*MJ)$. Conversely, suppose that $v^TMv<0$. Since $J$ has a trivial kernel, there exists $u$ such that $v=Ju$, and therefore $0>v^TMv=(Ju)^TMJu=u^TJ^*MJu$ which implies that $u$ is in the negative eigenspace of $J^*MJ$.

When carrying out the calculations involved in the conjugation $J^*MJ$, and requiring that the off-diagonal terms be $0$, one finds:
	\begin{align*}
	J_1
	&=
	\left(A_2-DA_3^{-1}D^*\right)^{-1}\left(-B^*+DA_3^{-1}C^*\right)\\
	J_2
	&=
	-A_3^{-1}D^*J_1^*-A_3^{-1}C^*\\
	J_3
	&=
	-A_3^{-1}D^*,
	\end{align*}
and
	\begin{align*}
	\Delta_1
	&=
	A_1-\left(B-CA_3^{-1}D^*\right)\left(A_2-DA_3^{-1}D^*\right)^{-1}\left(B^*-DA_3^{-1}C^*\right)-CA_3^{-1}C^*\\
	\Delta_2
	&=
	A_2-DA_3^{-1}D^*\\
	\Delta_3
	&=
	A_3.
	\end{align*}

Returning to our matrix operator $\Ml$, we define two projection operators:
\begin{defi}
\begin{enumerate}
\item
Let $P_n$ be the orthogonal projection onto the eigenspace associated with the first $n$ eigenvalues (counting multiplicity) of $\Aonez$.
\item
Let $Q_n$ be the orthogonal projection onto the eigenspace associated with the first $n$ eigenvalues (counting multiplicity) of $\Atwoz$.
\end{enumerate}
\end{defi}
Then we define the \emph{truncated matrix operator} to be
	\be
	\Mnl
	=
	\matthree{-\Aoneln}{\Bln}{\Cln}{\left(\Bln\right)^*}{\Atwoln}{-\Dln}{\left(\Cln\right)^*}{-\left(\Dln\right)^*}{-P\left(\l^2-l^\l\right)}
	\en
where
	\begin{align*}
	\Aoneln=P_n\Aonel P_n  \hspace{1cm}  \Bln=P_n\Bl Q_n  \hspace{1cm}  \Cln=P_n\Cl\\
	\Atwoln=Q_n\Atwol Q_n  \hspace{3.9cm}  \Dln=Q_n\Cl.
	\end{align*}
When $\l=0$ the truncated matrix operator becomes
	\be
	\Mnz
	=
	\matthree{-\Aonezn}{\Bzn}{0}{\left(\Bzn\right)^*}{\Atwozn}{0}{0}{0}{Pl^0}.
	\en

We wish to apply the diagonalization technique to $\Mnz$, yet we do not wish to invert the operator $\Atwoz$. We therefore apply our technique to the operator that is the result of switching the first two rows and columns of $\Mnz$. After diagonalizing, we obtain the matrix operator
	\be
	\Fnz=\matthree{\Konezn}{0}{0}{0}{-\Aonezn}{0}{0}{0}{Pl^0}
	\en
where $\Konezn=\Atwozn+\left(\Bzn\right)^*\left(\Aonezn\right)^{-1}\Bzn$. Now we see the benefit of this diagonalization process. Instead of considering $\Mnz$ we consider $\Fnz$ which has the same number of negative eigenvalues, but for which the actual count is simpler:
	\begin{align}\label{diag-l0}
	neg\left(\Mnz\right)
	&=neg\left(\Atwozn+\left(\Bzn\right)^*\left(\Aonezn\right)^{-1}\Bzn\right)+neg\left(-\Aonezn\right)+neg\left(l^0\right)\\
	&=
	neg\left(\Atwozn+\left(\Bzn\right)^*\left(\Aonezn\right)^{-1}\Bzn\right)+n-z\left(\Aonezn\right)-neg\left(\Aonezn\right)+neg\left(l^0\right).\nonumber
	\end{align}
Lemma \ref{well-defined} makes this expression rigorous (that is, in Lemma \ref{well-defined} we show that one can invert the operator $\Aonez$ when composed with $\Bz$).

In the next section we shall see the relationship between $\Mnz$ and $\Mnl$ when $\l$ is small (see \eqref{third-eq}). This analysis will prove to be very simple, due to the results of \S\ref{func-anal-back}. We now turn to the analysis of $\Mnl$ when $\l$ is large. This analysis also turns out to be simple, mostly due to the appearance of the term $\l^2$ in two crucial locations.
%
%
%
%
\begin{lemma}\label{large-la}
There exists $\Lambda^*>0$ such that for any $n\in\N$ and any $\l>\Lambda^*$, $\Mnl$ has exactly $n+1$ negative eigenvalues.
\end{lemma}

\begin{proof}
Since $\Mnl$ is a symmetric mapping on a finite-dimensional subspace of $\hptz\times\hpt\times\R$ that is $2n+1$-dimensional, it has $2n+1$ eigenvalues, all real. Letting $\psi\in\hpt$, we have

	\begin{equation}
	\left<\Mnl \left(\begin{array}{c}0\\\psi\\0\end{array}\right),\left(\begin{array}{c}0\\\psi\\0\end{array}\right)\right>_{\lpt\times\lpt\times\R}
	=
	\left<\Atwol Q_n\psi,Q_n\psi\right>_{\lpt}
	>
	0
	\end{equation}
for all $\la>\La$ by Lemma \ref{properties1}. This implies that $\Mnl$ is positive definite on a subspace of dimension $n$, and, therefore it has at least $n$ positive eigenvalues. Similarly, we now show that there exists a subspace of dimension $n+1$ on which $\Mnl$ is negative definite: Let $(\phi,0,b)\in\hptz\times\hpt\times\R$ and consider
	\begin{align}\label{neg-def}
	\left<\Mnl \left(\begin{array}{c}\phi\\0\\b\end{array}\right),\left(\begin{array}{c}\phi\\0\\b\end{array}\right)\right>_{\lpt\times\lpt\times\R}
	&=
	-\left<\Aonel P_n\phi,P_n\phi\right>_{\lpt}
	+
	2\left<\Cl b,P_n\phi\right>_{\lpt}
	-
	P(\la^2-l^\la)b^2.
	\end{align}
We estimate the second term:
	\begin{align*}
	2\left|\left<\Cl b,P_n\phi\right>_{\lpt}\right|
	\leq
	2\left\|\Cl b\right\|_{\lpt}\left\|P_n\phi\right\|_{\lpt}
	\leq
	\frac{\left\|\Cl b\right\|^2_{\lpt}}{\epsilon^2}+\epsilon^2\left\|P_n\phi\right\|^2_{\lpt}.
	\end{align*}
Letting $\epsilon^2=\frac{1}{\la}$, we have
	\begin{align*}
	\left<\Mnl \left(\begin{array}{c}\phi\\0\\b\end{array}\right),\left(\begin{array}{c}\phi\\0\\b\end{array}\right)\right>_{\lpt\times\lpt\times\R}
	\leq
	-\left<\Aonel P_n\phi,P_n\phi\right>_{\lpt}
	+
	\frac{\left\|P_n\phi\right\|^2_{\lpt}}{\la}
	-
	P(\la^2-l^\la)b^2
	+
	\la\left\|\Cl b\right\|^2_{\lpt}.
	\end{align*}
By Lemma \ref{properties1}, $\Aonel>\gamma>0$ for all $\la$ large enough, and therefore this expression is negative for all $\phi\in\hptz$ and $b\in\R$, since $l^\la$ and $\Cl$ are both bounded. Therefore, there exists  a $\La^*>0$ such that for every $\la\geq\La^*$ there exists an $n+1$ dimensional subspace and on which $\Mnl$ is negative definite. We conclude that
	\begin{equation}
	neg\left(\Mnl\right)=n+1,
	\hspace{.5cm}
	\text{for all }\la>\La^*.
	\end{equation}
Notice that $\La^*$ does not depend upon $n$.
\end{proof}

\section{The Limit $n\to\infty$}\label{limit-section}
\begin{lemma}\label{limit-lemma}
There exists $N_1>0$ such that for all $n>N_1$ it holds that $neg\left(\Aonezn\right)=neg\left(\Aonez\right)$ and $neg\left(\Atwozn+\left(\Bzn\right)^*\left(\Aonezn\right)^{-1}\Bzn\right)=neg\left(\Atwoz+\left(\Bz\right)^*\left(\Aonez\right)^{-1}\Bz\right)$.
\end{lemma}

\begin{proof}
By Lemmas \ref{properties1} and \ref{k-op} we know that both $\Aonez$ and $\Atwoz+\left(\Bz\right)^*\left(\Aonez\right)^{-1}\Bz$ are operators of the type discussed in \S\ref{func-anal-back}, that is, they are both of the form $-\triangle+K$, where $K$ is a bounded operator. Therefore, both operators have finitely many negative eigenvalues. By the definitions of the projection operators $P_n$ and $Q_n$ it is clear that for $n$ that is sufficiently large the truncated operators will recover the negative eigenspace of the original operators.
\end{proof}

\begin{lemma}\label{approx-matrix-op-lemma}
\begin{enumerate}
\item
Under the hypothesis \eqref{mainthm-condition} of Theorem \ref{mainthm}, there exists $N>0$ such that for every fixed $n>N$, there exists $0<\l_n<\infty$ such that $\Mnln$ has a nontrivial kernel. Moreover, the (real) numbers $\l_n$ are bounded uniformly away from $0$ and $\infty$.
\item
The same conclusion holds under the hypothesis \eqref{mainthm2-condition} of Theorem \ref{mainthm2}.
\end{enumerate}
\end{lemma}

\begin{proof}
\begin{enumerate}
\item
We first show that the statement holds under the hypothesis \eqref{mainthm-condition} imposed in Theorem \ref{mainthm}, namely
	$$
	neg\left(\Atwoz+\left(\Bz\right)^*\left(\Aonez\right)^{-1}\Bz\right)> neg\left(\Aonez\right)+neg(-l^0),
	$$
as well as the assumption that $\ker\Aonez$ contains only the constant functions and that $l^0\neq0$. Therefore this inequality may be rewritten as
	\be\label{first-eq}
	neg\left(\Atwoz+\left(\Bz\right)^*\left(\Aonez\right)^{-1}\Bz\right)-neg\left(\Aonez\right)-1+neg(l^0)+n>n.
	\en
	
Now, by \eqref{diag-l0}, by Lemma \ref{limit-lemma} and using the fact that $z\left(\Aonez\right)=0$ on $\hptz$ (recalling that $\hptz$ does not include the constant functions),
	\begin{align}
	neg\left(\Mnz\right)
	&=
	neg\left(\Atwozn+\left(\Bzn\right)^*\left(\Aonezn\right)^{-1}\Bzn\right)+n-z\left(\Aonezn\right)-neg\left(\Aonezn\right)+neg\left(l^0\right)\label{second-eq}\\
	&=
	neg\left(\Atwoz+\left(\Bz\right)^*\left(\Aonez\right)^{-1}\Bz\right)+n-neg\left(\Aonez\right)+neg\left(l^0\right)\nonumber
	\end{align}
for all $n>N_1$ (where $N_1$ is given in Lemma \ref{limit-lemma}).
By Proposition \ref{second-prop2}, there exist $\l^*$ and $N_2$ such that for all $0<\l\leq\l^*$ and for all $n>N_2$
	\be\label{third-eq}
	neg(\Mnl)\geq neg(\Mnz).
	\en
By Lemma \ref{large-la} there exists $\Lambda^*$ such that for all $n$ and for all $\l\geq\Lambda^*$
	\be\label{fourth-eq}
	neg(\Mnl)=n+1.
	\en

Finally, combining \eqref{first-eq}-\eqref{fourth-eq}, we have
	\begin{align*}
	neg(\mathcal{M}_n^{\l^*})
	&\geq
	neg(\Mnz)\\
	&=
	neg\left(\Atwoz+\left(\Bz\right)^*\left(\Aonez\right)^{-1}\Bz\right)+n-neg\left(\Aonez\right)+neg\left(l^0\right)\\
	&>
	n+1\\
	&=
	neg(\mathcal{M}_n^{\Lambda^*})
	\end{align*}
for all $n>N:=\max(N_1,N_2)$. Since $\Mnl$ is a finite-dimensional mapping, for fixed $n$ its spectrum, which is real, varies continuously (as a set) as $\l$ varies. Hence at least one eigenvalue must cross $0$ from left to right as $\l$ varies between $\l^*$ and $\Lambda^*$. Denoting a value of $\l$ for which $\Mnl$ has a nontrivial kernel by $\l_n$ (noting that it depends on $n$), we deduce that
	\be\label{approx-kernel}
	\Mnln u_n=0
	\en
for some $u_n\neq0$, and $0<\l^*<\l_n<\Lambda^*<\infty$.\\

\item
In the context of Theorem \ref{mainthm2}, \eqref{first-eq} becomes
	\be\label{first-eq2}
	neg\left(\Atwoz+\left(\Bz\right)^*\left(\Aonez\right)^{-1}\Bz\right)-neg\left(\Aonez\right)-1+neg(l^0)+n\neq n.
	\en
Using Proposition \ref{first-prop2}, \eqref{third-eq} becomes
	\be\label{third-eq2}
	neg(\Mnl)= neg(\Mnz)
	\en
for all $0\leq\l\leq\l^*$.
Therefore, combining \eqref{second-eq}, \eqref{fourth-eq}, \eqref{first-eq2}, \eqref{third-eq2}, 
	\begin{align*}
	neg(\mathcal{M}_n^{\l^*})
	&=
	neg(\Mnz)\\
	&=
	neg\left(\Atwoz+\left(\Bz\right)^*\left(\Aonez\right)^{-1}\Bz\right)+n-neg\left(\Aonez\right)+neg\left(l^0\right)\\
	&\neq
	n+1\\
	&=
	neg(\mathcal{M}_n^{\Lambda^*})
	\end{align*}
and therefore the conclusion remains: there exists $0<\l_n<\infty$ for which $\Mnln$ has a nontrivial kernel.
\end{enumerate}
\end{proof}

Now we make use of \S\ref{func-anal-back} to reach a conclusion about the untruncated operator:
\begin{lemma}\label{untrunc-problem}
There exists $0<\l_0<\infty$ and some $u^T_0=(\phi_0,\psi_0,b_0)\neq(0,0,0)$ such that
	\be
	\Mlz u_0=0.
	\en
\end{lemma}

\begin{proof}
By Lemma \ref{approx-matrix-op-lemma} there exist $\l_n$ bounded uniformly away from $0$ and $\infty$ ($0<\l^*<\l_n<\Lambda^*<\infty$), that satisfy \eqref{approx-kernel} with $u_n\neq0$. Therefore there exists a subsequence of $\{\l_n\}_{n=N}^\infty$ that converges to some $0<\l_0<\infty$. The existence of $u_0\neq0$ is guaranteed by Proposition \ref{prop:limit}.
\end{proof}

\begin{proof}[Proof of Theorems \ref{mainthm} and \ref{mainthm2}]\label{proof:theorems}
In \S\ref{setup-section} we showed that finding a growing mode solution of the form \eqref{purely} is equivalent to finding a nontrivial kernel for $\Ml$ for some $0<\l<\infty$. Using the tools developed in \S\ref{func-anal-back}, in \S\ref{limit-section} we show that, indeed, there exists such a value of $\l$. The last ingredient of the proof, which we omit here and which can be found in \cite[\S5]{benartzi1} and \cite[\S8]{rvm2} is to verify that one can indeed construct a growing mode from the potentials $\phi_0,\psi_0,b_0$ that we have found in Lemma \ref{untrunc-problem}. This part is relatively straight-forward, as one can define $f^\pm,E,B$ from the potentials, and then verify that they satisfy the linearized system.
\end{proof}

\section{The Operators}\label{the-operators}
In this section we state the important properties of the operators appearing in this paper. Most of these properties are proved in detail in \cite{benartzi1}. We only discuss the proofs of new properties, notably Lemmas \ref{well-defined} and \ref{k-op}.
\begin{lemma}[Properties of $\Aonel, \Atwol$]\label{properties1}
Let $0\leq\la<\infty$.
	\begin{enumerate}
	\item\label{op-norm}
	$\Aonel$ is selfadjoint on $\lptz$. $\Atwol$ is selfadjoint on $\lpt$. Their domains are $\hptz$ and $\hpt$, respectively, and their spectra are discrete.
	
	\item
	For all $h(x)\in\hptz$, $\|\Aonel h-\Aonez h\|_{\lpt}\to0$ as $\la\to0$. The same is true for $\Atwol$ with $h(x)\in\hpt$.

	\item
	For $i=1,2$ and $\sigma>0$, it holds that $\|\Ail-\Ais\|=O(|\la-\sigma|)$ as $\la\to\sigma$, where $\|\cdot\|$ is the operator 	norm from $\hptz$ to $\lpt$ in the case $i=1$, and from $\hpt$ to $\lpt$ in the case $i=2$.

	\item
	For all $h(x)\in\hptz$, $\|\Aonel h+\p_x^2h\|_{\lpt}\to0$ as $\la\to\infty$.

	\item
	When thought of as acting on $\hpt$ (rather than $\hptz$), the null spaces of $\Aonel$ and $\Aonez$ both contain the constant functions.
	
	\item\label{operators-positive}
	There exist constants $\gamma>0$ and $\La>0$ such that for all $\la\geq\La$, $\Ail>\gamma>0$, $i=1,2$.
	\end{enumerate}
\end{lemma}


\begin{lemma}[Properties of $\Bl, \Cl, \Dl$]\label{properties2}
Let $0<\la<\infty$.

	\begin{enumerate}
	\item\label{blnorm}
	$\Bl$ maps $\lpt\to\lpt$ with operator bound independent of $\la$. Moreover, $R(\Bz)\subseteq\{1\}^\perp$.
	
	\item\label{bz}
	For all $h(x)\in\lpt$, as $\la\to0$ we have: $\|\Bl h-\Bz h\|_{\lpt}\to0$ and $\|\Cl h\|_{\lpt},\|\Dl h\|_{\lpt}\to0$.

	\item
	If $\sigma>0$, then $\|\Bl-\Bs\|=O(|\la-\sigma|)$ as $\la\to\sigma$, where $\|\cdot\|$ is the operator norm from $\lpt$ to $\lpt$. The same is true for $\Cl, \Dl$.

	\item\label{bl-infty}
	For all $h(x)\in\lpt$, $\|\Bl h\|_{\lpt}\to0$ as $\la\to\infty$. The same is true for $\Cl, \Dl$.
	
	\end{enumerate}
\end{lemma}

\begin{proof}
This proof is presented almost in its entirety in \cite{benartzi1}, except for the statements of parts \ref{blnorm} and \ref{bz}. In part \ref{blnorm} we have added the statement about the range of $\Bz$. It is easily verified that $\int_0^P\Bz h\;dx=0$ (see \cite[Lemma 2.4]{rvm1}). As for part \ref{bz} of the lemma: in \cite{benartzi1} we showed that $\Bl\xrightarrow{s}\Bz=0$. In the current setting, it is not true that $\Bz$ is trivial, however, it still \emph{is} true that $\Bl\xrightarrow{s}\Bz$.
\end{proof}

\begin{lemma}[Properties of $l^\la$]\label{propll}
Let $0<\la<\infty$.
	\begin{enumerate}
	\item
	$l^\la\to l^0$ as $\la\to0$.
	
	\item
	$l^\la$ is uniformly bounded in $\la$.
	
	\end{enumerate}
\end{lemma}

The following lemma lists the important properties of $\Ml$ -- all of which are inherited directly from the properties of the various operators it is made up of, as listed in Lemma \ref{properties1}. 
\begin{lemma}[Properties of $\Ml$]\label{propml}
To simplify notation, we write $u^T$ for a generic element $\left(\phi,\psi,b\right)\in\hpt\times\hpt\times\R$.
\begin{enumerate}
\item
For all $\la\geq0$, $\Ml$ is selfadjoint on $\lpt\times\lpt\times\R$ with domain $\hpt\times\hpt\times\R$.

\item
For all $u^T\in\hpt\times\hpt\times\R$, $\|\Ml u-\Mz u\|_{\lpt\times\lpt\times\lpt}\to0$ as $\la\to0$.

\item
If $\sigma>0$, then $\|\Ml-\Ms\|\to0$ as $\la\to\sigma$, where $\|\cdot\|$ is the operator norm from $\hptz\times\hpt\times\R$ to $\lpt\times\lpt\times\lpt$.

\end{enumerate}
\end{lemma}

\begin{lemma}\label{well-defined}
The operator $\left(\Bz\right)^*\left(\Aonez\right)^{-1}\Bz$ is a well-defined bounded operator from $\lpt\to\lpt$.
\end{lemma}

\begin{proof}
By part \ref{blnorm} of Lemma \ref{properties2}, $R(\Bz)\subseteq\{1\}^\perp$. By assumption, the null space of $\Aonez$ contains only the constant functions. Since the spectrum of $\Aonez$ is discrete, it is invertible on $\{1\}^\perp$, and, therefore, $\left(\Bz\right)^*\left(\Aonez\right)^{-1}\Bz$ is a well-defined operator. This operator is bounded since $\Bz$ is bounded (again, see part \ref{blnorm} of Lemma \ref{properties2}) and since $\left(\Aonez\right)^{-1}$ is bounded.
\end{proof}

\begin{lemma}\label{k-op}
The operator $\Konez=\Atwoz+\left(\Bz\right)^*\left(\Aonez\right)^{-1}\Bz$ is of the form presented in \eqref{schrodinger-op}.
\end{lemma}

\begin{proof}
This is clearly true by the properties of $\Atwoz$ and by Lemma \ref{well-defined}.
\end{proof}

\end{document}